\documentclass[10pt]{amsart}

\usepackage{dsfont,mathptmx,txfonts}

\usepackage[all]{xy}

\usepackage[dvipsnames]{xcolor}   
\usepackage{xparse}
\usepackage{xr-hyper}
\usepackage{tikz-cd}
\usepackage{multirow,enumitem,array,longtable}
\usepackage{accents}
\usepackage{pst-node,placeins,xspace,fix-cm}
\usepackage[linktocpage=true,colorlinks=true,hyperindex,citecolor=teal,linkcolor=blue]{hyperref}

\newcommand{\dfn}[1]{\textit{#1}}
\newcommand{\cont}{\subseteq}
\newcommand{\iso}{\cong}
\newcommand{\End}[2][]{\operatorname{End}_{#1}\left( #2\right)}
\newcommand{\Mod}[1]{{#1}\mbox{-}\operatorname{Mod}}
\newcommand{\sdprod}[1][]{\hspace{0.3ex}{\rtimes}_{#1}\hspace{0.3ex}}
\newcommand{\Max}[1]{\mathrm{Max}\left({#1}\right)}
\newcommand{\m}{\mathfrak{m}}
\newcommand{\U}{\operatorname{U}}
\newcommand{\aar}[2][]{\xrightarrow[#1]{#2}}
\def\set#1#2{\left\{{#1}\left.\right|\,{#2}\right\}}

\newtheorem{theorem}{Theorem}[section]
\newtheorem{proposition}[theorem]{Proposition}
\newtheorem{lemma}[theorem]{Lemma}
\newtheorem{corollary}[theorem]{Corollary}
\theoremstyle{definition}
\newtheorem{definition}{Definition}[section]
\newtheorem{remark}[theorem]{Remark}
\newtheorem{example}[theorem]{Example}
\numberwithin{equation}{section}
\newtheorem{question}[theorem]{Problem}

\numberwithin{equation}{section}

\begin{document}

\title{On extensions of Cohen Structure Theorem}

\author{Elena Caviglia}
\address{(Elena Caviglia) [1] Department of Mathematics, Stellenbosch University, South Africa. [2]  National Institute for Theoretical and Computational Sciences (NITheCS), South Africa.}
\email{elena.caviglia@outlook.com}

\author{Amartya Goswami}

\address{(Amartya Goswami) [1] Department of Mathematics and Applied Mathematics, University of Johannesburg, P.O. Box 524, Auckland Park 2006, South Africa. [2]  National Institute for Theoretical and Computational Sciences (NITheCS), South Africa.}

\email{agoswami@uj.ac.za}

\author{Luca Mesiti}
\address{(Luca Mesiti) Department of Mathematics, University of KwaZulu-Natal, South Africa.}
\email{luca.mesiti@outlook.com}






\begin{abstract}
The aim of this paper is to extend Cohen structure theorem beyond local rings. Both Cohen structure theorem and Nagata's generalization of it are special cases of our results. We investigate for which rings 
$R$ there exists a maximal ideal 
$\m$ of 
$R$ such that the canonical projection 
$R\to R/\m$ has a section, so that 
$R/\m$ is isomorphic to a field 
$\kappa$ contained in 
$R$. We present two equivalent characterizations of this property and use them to exhibit two classes of rings that satisfy it. Moreover, we provide several examples (not necessarily local or complete local), as well as methods to construct new examples.
\end{abstract}

\maketitle 

\section{Introduction}

Let $A$ be a local Noetherian ring with maximal ideal $\mathfrak{m}$. The ring $A$ is said to be \emph{equicharacteristic}
 if $A$ has the same characteristic as its residue field $A/\mathfrak{m}$. A \emph{field of representatives} for $A$
is a subfield $\kappa$ in $A$ that maps onto $A/\mathfrak{m}$ under the canonical projection. Since
$\kappa$ is a field, the restriction of this mapping to $\kappa$ gives an isomorphism between $\kappa$ and $A/\mathfrak{m}$.
The
Cohen structure theorem  says that

\begin{theorem}\label{C46}
Every equicharacteristic complete local Noetherian commutative ring $A$ admits a field of representatives.
\end{theorem}

Over the years, this pivotal result in commutative algebra has proved to be significant--for example, in proving category equivalence (see \cite[Proposition 1.1]{GN01}) between the
category of all finite $R$-modules of finite projective and all finite $R$-modules of finite injective dimension for noncommutative Cohen–Macaulay local rings $R$; in application (see \cite[Theorem 3.1]{GH97}) of asserting  the existence of a local extension ring $S$  of local Noetherian $R$ under certain conditions on $S$; in proving a result on $2$-step solvable groups of finite Morley rank (see \cite[Section 3]{EN90}); in showing the existence of a local field (see \cite[Theorem 5]{LW73}) of a topological
quotient
ring $A$ (with characteristic
either
zero or a prime) of a commutative,
special
basic
ring; in  characterizing  a special local ring (see \cite[Theorem 2]{Kol68}) which is
Hausdorff with respect to the $U$-adic topology; in theorems (see \cite{Zar47}) on quotient rings of simple subvarieties; in constructing functors (see \cite[Theorem 1.4]{Bor04}) between the category of complete residual
perfections of a ring and  the full subcategory of the category of $\mathds{F}_p$-algebras whose
objects are perfect; in showing the existence of a functor (see \cite[Proposition 3.1.4]{Ked07}) from based fields to based Cohen rings which
is a quasi-inverse of the residue field functor; in showing a $K$-isomorphism (see \cite[Proposition 5.3]{GR09}) of a complete $n$-discrete $k$-field containing a perfect subfield $K$ (with trivial valuation); in establishing the reciprocity law (see \cite[Lemma 2.8]{Mor12}) along a vertical curve for residues of differential
forms on arithmetic surfaces; showing well-definiteness of a family of invariants (see \cite{N-BW13}) associated to any local ring whose residue
field has prime characteristic; in Dieudonné theory (see \cite{CL17}) over complete regular local rings; in passaging from an order-by-order solution to a formal solution (see \cite[Theorem 2.4]{BGK22}); in proving an embedding lemma (see \cite[Lemma 4.6]{ADJ24}) for $\mathds{Z}$-valued fields.

Apart from its numerous applications, as discussed above, significant research has been dedicated to the structure theorem itself. In \cite{Coh46}, local ring actually means a commutative unitary Noetherian ring with a unique maximal ideal; whereas in \cite{Nag50} Nagata obtained a similar structural result weakening the Noetherian condition. The proof of the latter result was further simplified in \cite{Nar55}. On the other hand, a weaker version of Cohen structure theorem has been proved for any local ring (not necessarily complete or Noetherian) in \cite{Mat77}.

Our aim in this paper is to extend Cohen structure theorem, studying which rings $R$ have the following property:
\smallskip

$(\star)\;$ \textit{
There exists a maximal ideal $\m$ of $R$ such that the canonical projection $R\aar{\pi} R/\m$ has a section, so that $R/\m$ is isomorphic to a field $\kappa$ contained in $R$.
}
\smallskip

In other words, $(\star)$ asks $R$ to contain a subfield $\kappa$ and a maximal ideal $\m$ such that $R/\m\iso \kappa$, in a way that gives a section to the canonical projection $R\to R/\m$. Of course $R/\m$ is always a field, but we are interested in when this field is a subring of $R$. Notice that property $(\star)$ is stronger than having
\[R/\m\iso \kappa\cont R,\]
as we want the ring homomorphism $R/\m\iso\kappa\cont R$ to be a section of the canonical projection $R\to R/\m$. 
For a local Noetherian ring, property $(\star)$ coincides with having a field of representatives. So by Theorem \ref{C46}, every equicharacteristic complete local Noetherian commutative ring satisfies property $(\star)$. For example, the ring of formal power series \( \kappa\llbracket x_1, \ldots, x_n\rrbracket \) over a field \( \kappa \) with maximal ideal $\langle x_1,\ldots, x_n\rangle$. And of course a trivial example is any field $\kappa$. Moreover, notice that if a ring $A$ contains a field $\kappa$, then it is equicharacteristic.

However, there exist numerous rings that satisfy property 
$(\star)$ without being complete local or even local (see Section \ref{mex}). Our aim is to gain a deeper understanding of the class of rings—much broader than those satisfying the assumptions of Cohen's structure theorem—that satisfy property 
$(\star)$.

We show how Cohen structure theorem as well as Nagata's generalization are special cases of our result (see Theorem \ref{theorAB}). We first prove two equivalent characterizations (Theorem \ref{theorsdprod} and Theorem \ref{theorsum}) of property $(\star)$, which will then help us check property $(\star)$ more easily. In fact, usually way more easily than checking the completeness property required in Cohen structure theorem. We achieve our characterizations considering a semidirect product of rings. Moreover, we produce several methods to construct examples of rings that satisfy property $(\star)$ (see for example Proposition \ref{propsr} and Proposition \ref{proplochasstar}). We present numerous examples of different kinds of rings that satisfy property $(\star)$.

To set up the ground, let us settle with the  notation and terminologies. Unless otherwise specified, all rings considered in this paper will be unitary with identity element denoted as 1, and all ring homomorphisms will preserve 1. All the ideals we consider will be two-sided. We denote the set of maximal ideals in a ring $R$ by $\mathrm{Max}(R),$ whereas the notation $\mathfrak{m}^*$  denotes the set of all non-zero elements in a maximal ideal $\mathfrak{m}.$ By $U(R)$, we denote the set of invertible elements in $R$. If $\kappa$ is a field, then the set of all invertible elements in $\kappa$ is denoted by $\kappa^*$. When $\kappa\subseteq A$, then we sometimes represent it as the canonical inclusion map $\iota\colon \kappa\to A,$ and we  write  $\iota(\kappa)$. Such a ring $R$ can be seen as a $\kappa$-algebra. Whenever we say `$R$ is a $\kappa$-algebra', we mean that there exists a copy of the field $\kappa$ included in $R$.

\section{Motivating examples}\label{mex}

In this section, we present several motivating examples that led us to study the problem of interest.

\begin{example}
    Rings of polynomials \( \kappa[x_1, \dots, x_n] \) with coefficients in a field $\kappa$ clearly satisfy property $(\star)$. Indeed, \(\langle x_1, \dots, x_n\rangle\) is a maximal ideal. Moreover the ring homomorphism
    $$\kappa[x_1, \dots, x_n]\to \kappa$$
    that sends every polynomial $p$ to its degree 0 coefficient $p(0)$ is surjective, and its kernel is precisely $\langle x_1, \dots, x_n\rangle$. By the first isomorphism theorem, we obtain that 
    $$\kappa[x_1, \dots, x_n]/\langle x_1, \dots, x_n\rangle\iso \kappa.$$
    And, of course, a copy of $\kappa$ is included in $\kappa[x_1, \dots, x_n]$. It is easy to check that  \[\kappa[x_1, \dots, x_n]/\langle x_1, \dots, x_n\rangle\iso \kappa\cont \kappa[x_1, \dots, x_n]\] gives a section to the canonical projection $\kappa[x_1, \dots, x_n]\to \kappa[x_1, \dots, x_n]/\langle x_1, \dots, x_n\rangle$. 
\end{example}

\begin{remark}
    Although \( \mathds{R}[x] \) satisfies property $(\star)$, exhibited by $\langle x\rangle\in \Max{\mathds{R}[x]}$ and the field $\mathds{R}\cont \mathds{R}[x]$, not every maximal ideal of $\mathds{R}[x]$ works. Indeed, $\langle x^2+1\rangle\in \Max{\mathds{R}[x]}$,  but $\mathds{R}[x]/\langle x^2+1\rangle\iso \mathds{C}\not\subseteq\mathds{R}$.
\end{remark}

The following example shows that there are non-radical rings that also satisfy property $(\star)$.

\begin{example}
    Consider the quotient ring $R=\kappa[x]/\langle x^2\rangle$, where $\kappa$ is a field. The ideal $\langle [x]\rangle$ is maximal in $R$ and the quotient $R / \langle [x]\rangle$ is isomorphic to $\kappa$. Indeed, the ring homomorphism that sends $[p]\in R$ to the degree zero coefficient of the polynomial $p$ is surjective and has kernel $\langle [x]\rangle$. 
    Moreover, it is easy to check that this ring homomorphism gives a section to the canonical projection $R \to R/ \langle [x]\rangle$. So $R$ has property $(\star)$. But $R$ is not radical as $[x]$ is a non-zero nilpotent element of $R$.
\end{example}

The following is another example of a class of rings that satisfy property $(\star)$, although being of different nature.

\begin{example}\label{exfunctions}
    Let $\kappa$ be a field and let $X$ be a set. Consider the ring \( \operatorname{Hom}(X,\kappa) \) of all functions from $X$ to $\kappa$, with sum and product defined pointwise. This is a commutative unitary ring, thanks to the fact that $\kappa$ is so. A copy of $\kappa$ is embedded in  \( \operatorname{Hom}(X,\kappa) \), given by the constant functions. Let $a\in X$ and consider the set
    $$\m_a:=\set{f\in \operatorname{Hom}(X,\kappa)} {f(a)=0}.$$
     Clearly $\m_a$ is an ideal of \( \operatorname{Hom}(X,\kappa) \). Moreover the ring homomorphism
    \[\operatorname{Hom}(X,\kappa)\to \kappa\]
   defined by $(f\colon X\to \kappa)\mapsto f(a)$ is surjective 
   and has kernel $\m_a$. Then, by the first isomorphism theorem,
    \[\operatorname{Hom}(X,\kappa)/\m_a\iso \kappa.\]
    As $\kappa$ is a field, this also implies that $\m_a$ is a maximal ideal. Finally, notice that the homomorphism \[\operatorname{Hom}(X,\kappa)/\m_a\iso \kappa\cont \operatorname{Hom}(X,\kappa)\] gives a section to the canonical projection. Indeed, given a function $f\colon X\to \kappa$, we have that $f=\operatorname{const}_{f(a)}+(f-\operatorname{const}_{f(a)})$. As $f-\operatorname{const}_{f(a)}\in \m_a$, we obtain that $[f]=[\operatorname{const}_{f(a)}]$.
\end{example}

\begin{remark}\label{remuniquenessmk}
    Different maximal ideals $\m$ of $\operatorname{Hom}(X,\kappa)$ exhibit, together with $\kappa$, that $\operatorname{Hom}(X,\kappa)$ has property $(\star)$. So the maximal ideal $\m$ involved in property $(\star)$ is not unique in general. Fixed $\m$, though, there can only be one field $\kappa$ up to isomorphism that makes property $(\star)$ hold, which is $R/\m$.
\end{remark}

We now show an example of a local $\kappa$-algebra $(A, \mathfrak{m})$ such that $\kappa$ is a maximal field contained in $A$, but which does not satisfy property $(\star)$.

\begin{example}
Consider the ring $\mathds{Q}[x]$ of polynomials with rational coefficients and its localization at its maximal ideal $\langle x^2 + 1\rangle $:
\[
R := \mathds{Q}[x]_{\langle x^2+1 \rangle} = \set{\dfrac{f(x)}{g(x)} \in \mathds{Q}(x)}{ (x^2 + 1) \text{ does not divide } g(x)} \subset \mathds{Q}(x).
\]
Then $R$ is a local ring whose residue field is isomorphic to $\mathds{Q}[x]/\langle x^2 + 1\rangle \cong \mathds{Q}(i)$ (where $i^2 = -1$). Indeed, the function
\[\mathds{Q}[x]_{\langle x^2 + 1\rangle}\to \mathds{Q}[x]/\langle x^2 + 1\rangle\quad \left(\dfrac{f(x)}{g(x)}\mapsto [f(x)]\cdot [g(x)]^{-1}\right)\]
 is a surjective ring homomorphism with kernel  $\langle x^2+1\rangle \mathds{Q}[x]_{\langle x^2 + 1\rangle}$. The desired isomorphism follows from the first isomorphism theorem.
Obviously $R$ does not contain a copy of $\mathds{Q}(i)$ in it. Notice that $R$ admits a maximal subfield, namely, $\mathds{Q}$. Indeed, let $\kappa$ be a field such that $\mathds{Q} \subseteq \kappa \subseteq R$. Then the canonical map	\[
R \to R/m \cong \mathds{Q}(i)
\]
sends $\kappa$ isomorphically to a subfield of $\mathds{Q}(i)$. Thus, up to isomorphism, we have $\mathds{Q} \subseteq k \subseteq \mathds{Q}(i)$, and since the extension $\mathds{Q} \subseteq \mathds{Q}(i)$ has degree 2 and $k \neq \mathds{Q}(i)$, we must have $k = \mathds{Q}$. Thus, $\mathds{Q}$ is a maximal subfield of $R$. It turns out that $R$ is a local $\mathds{Q}$-algebra, where $\mathds{Q}$ is a maximal subfield of $R$, but its residue field is not isomorphic to $\mathds{Q}$.
\end{example}

\begin{example}
Let $\kappa$ be a field and consider $R:=\kappa\times \kappa$. Then $R$ has a copy of $\kappa$ embedded in it, given by the diagonal. Indeed, the ring homomorphism
    $$\kappa\to \kappa\times \kappa$$
    that sends $a$ to $(a,a)$ preserves 1 and is injective. Notice that the pairs of the form $(a,0)$ with $a\in \kappa$ do not give a copy of $\kappa$ embedded in $\kappa\times \kappa$, as the corresponding injective group homomorphism $\kappa\to \kappa\times \kappa$ does not preserve 1. 
We have that $\kappa\times 0$ is a maximal ideal in $\kappa\times \kappa$. Indeed, in general all ideals of $S\times T$ (with $S,T$ rings) are products of an ideal in $S$ and an ideal in $T$. So if an ideal of $\kappa\times \kappa$ strictly contains $\kappa\times 0$, it needs to be the whole $\kappa\times \kappa$. The ring homomorphism $\kappa\times \kappa\to \kappa$ that sends $(a,b)$ to $b$ is surjective and its kernel coincides with $\kappa\times 0$. So by the first isomorphism theorem, $(\kappa\times\kappa)/(\kappa\times 0)\iso \kappa$. It is easy to check that \[(\kappa\times\kappa)/(\kappa\times 0)\iso \kappa\cont \kappa\times \kappa,\] and, this gives a section to the canonical projection. Indeed, given $(a,b)\in \kappa$, we have that $(a,b)=(b,b)+(a-b,0)$, whence $[(a,b)]=[(b,b)]$.
\end{example}

\section{Main results}\label{secmainresults}

To present our main results, we first establish two equivalent characterizations of property 
$(\star)$. These formulations will facilitate a simpler verification of property 
$(\star)$ for a given ring. In particular, for a local ring, our characterizations are often significantly easier to check than the completeness condition required in Cohen’s structure theorem. We identify two classes of rings that satisfy property 
$(\star)$, one of which strictly generalizes both the class of rings satisfying the assumptions of Cohen structure theorem and those satisfying Nagata’s assumptions. Furthermore, we provide a method for constructing examples of rings that satisfy property 
$(\star)$. In the next section, we shall present additional methods for constructing such examples.

Let us first recall the definition of semidirect products of rings.

\begin{definition}\label{defsdprod}
Let $B$ and $S$ be associative but not necessarily unitary rings. Let 
$$\lambda\colon S\to \End[\Mod{B}]{B_B}$$
be a ring homomorphism and
$$\rho \colon S\to \End[\Mod{B}]{_B B}$$
be a ring anti-homomorphism such that for every $s,$ $t\in S$ and $x,$ $y\in B$, we have
\[\lambda(s)\circ\rho(t)=\rho(t)\circ \lambda(s);\]
\[\rho(s)(x)\cdot y=x\cdot \lambda(s)(y).\] The \dfn{semidirect product} $B\sdprod[\lambda,\rho]S$ of the (not necessarily unitary) rings $B$ and $S$ with respect to $\lambda$ and $\rho$ is constructed as follows. Its underlying additive group is given by the direct sum $B\oplus S$. We then equip it with a product using $\lambda$ and $\rho$ as follows. Given $(b,s),(c,t)\in B\oplus S$, we define
\[(b,s)\cdot (c,t):=(b\cdot c+\lambda(s)(c)+\rho(t)(b),s\cdot t).\]
It is straightforward to check that this is a not necessarily unitary ring. Moreover, the injective group homomorphisms $i_B\colon B\to B\oplus S$ defined by $b\mapsto (b,0)$ and $i_S\colon S\to B\oplus S$ defined by $s\mapsto (0,s)$ preserve products, with respect to the product in $B\sdprod[\lambda,\rho] S$ defined above.
\end{definition}

\begin{remark}\label{remlambdarho}
If $S$ has 1 and both $\lambda$ and $\rho$ preserve 1, then also $B\sdprod[\lambda,\rho] S$ is unitary, with $1=(0,1)$. Moreover $i_S(1)=(0,1)=1$, so $i_S$ is an injective ring homomorphism and $S\cont B\sdprod[\lambda,\rho] S$.

For the rest of this paper, whenever $S$ has 1 and we consider  $B\sdprod[\lambda,\rho] S$, we implicitly assume that $\lambda$ and $\rho$ preserve 1, so that $B\sdprod[\lambda,\rho] S$ is a unitary ring containing a copy of $S$.
\end{remark}

The following theorem is inspired by a personal communication between the second author and George Janelidze.

\begin{theorem}\label{theorsdprod}
Let $R$ be a ring. The following two holds:
\begin{enumerate}
\item If there exists a not necessarily unitary ring $\m$, a field $\kappa$ and $\lambda,\rho$ as in Definition \ref{defsdprod} such that $R:=\m\sdprod[\lambda,\rho] \kappa$, then $\m\in\Max{R}$, $\kappa\cont R$ and $R/\m\iso \kappa$, in a way such that the homomorphism $R/\m\iso \kappa\cont R$ gives a section to the canonical projection, so that $R$ satisfies property $(\star)$.

\item If $R$ satisfies property $(\star)$, exhibited by $\m\in\Max{R}$ and a field $\kappa\cont R$, then $R\iso \m\sdprod \kappa$, with $\lambda$ and $\rho$ acting by multiplication in $R$, thanks to the inclusion $\kappa\cont R$.
\end{enumerate}
In other words, property $(\star)$ is equivalent to the ring being a semidirect product of a not necessarily unitary ring and a field. Moreover, we can always restrict to semidirect products with $\lambda$ and $\rho$ simply given by multiplication in $R$.
\end{theorem}
\begin{proof}
(1) Assume that $R:=\m\sdprod[\lambda,\rho] \kappa$, with $\kappa$ a field. By Remark \ref{remlambdarho}, $i_\kappa$ is an injective ring homomorphism, so that a copy of $\kappa$ is included in $R$. There is a natural projection map on the second component
    \[\pi_2\colon \m\sdprod[\lambda,\rho] \kappa\to \kappa.\]
    Of course, this is a surjective homomorphism of groups. By construction of the multiplication of the semidirect product, $\pi_2$ also preserves the multiplication ($\lambda$ and $\rho$ are only used in the first component), and we have that $\pi_2(0,1)=1$, so that $\pi_2$ preserves 1. The projection map $\pi_2$ is thus a surjective ring homomorphism.
    Notice  that \[\ker(\pi_2)=\set{(x,u)\in \m\sdprod[\lambda,\rho] \kappa}{u=0}\iso \m\] (as rings without $1$). So a copy of $\m$ is a two-sided ideal of $\m\sdprod[\lambda,\rho] \kappa$. Therefore, by the first isomorphism theorem, we obtain 
    \[(\m\sdprod[\lambda,\rho]\kappa) /\m\iso \kappa.\]
    Since $\kappa$ is a field, this also guarantees that $\m$ is a maximal ideal of $\m\sdprod[\lambda,\rho]\kappa$. Moreover, the ring homomorphism
    $$(\m\sdprod[\lambda,\rho]\kappa) /\m\iso \kappa\cont \m\sdprod[\lambda,\rho]\kappa$$
    gives a section to the canonical projection. Indeed, given $(x,u)\in \m\sdprod[\lambda,\rho]\kappa$, we have that $(x,u)=(0,u)+(x,0)$. Since $(x,0)$ belongs to the copy of $\m$, we conclude that $[(x,u)]=[(0,u)]$.

(2) Assume that there exist $\m\in \Max{R}$, a field $\kappa\cont R$ and an isomorphism $j\colon R/\m\iso \kappa$ such that the ring homomorphism $\iota\circ j\colon R/\m\iso \kappa\cont R$ gives a section to the canonical projection $\pi\colon R\to R/\m$. We can construct the semidirect product $\m\sdprod \kappa$, with $\lambda$ and $\rho$ acting by multiplication in $R$, thanks to the inclusion $\kappa\cont R$. More precisely, we define
    $$\lambda\colon\kappa\to \End[\Mod{\m}]{\m_{\m}}$$
    by sending $u\in \kappa$ to the endomorphism $x\in\m\mapsto u\cdot_R x\in \m$. Notice that
    \[\lambda(u)(x+y)=u\cdot_R(x+y)=u\cdot_R x + u\cdot_R y,\]
    \[\lambda(u)(x\alpha)=u\cdot_R(x\cdot_R\alpha)=(u\cdot_R x)\cdot_R\alpha.\]
    So $\lambda(u)$ is linear. Moreover,
    \begin{align*}
\lambda(u+v)(x)&=(u+v)\cdot_R x=u\cdot_R x+v\cdot_R x,    \\
\lambda(u\cdot v)(x)&=u\cdot_R v\cdot_R x=\lambda(u)(\lambda(v)(x)),\\
\lambda(1)&=\operatorname{id}.
   \end{align*}
Therefore, $\lambda$ is a ring homomorphism. Analogously, we define
    \[\rho\colon\kappa\to \End[\Mod{\m}]{_\m \m}\]
    by sending $u\in \kappa$ to the endomorphism $x\in\m\mapsto x\cdot_R u\in \m$, and this is a ring anti-homomorphism. It is easy to check that $\lambda$ and $\rho$ satisfy the compatibility axioms of Definition \ref{defsdprod}.
It remains to prove that $R\iso \m\sdprod \kappa$. Consider the function
    \[\phi\colon\m\sdprod \kappa\to R\]
    defined by $(x,u)\mapsto x+u$ (thanks to the inclusion $\kappa\cont R$). This is of course a group homomorphism and $\phi(0,1)=1$. Moreover
    \begin{align*}
  \phi((x,u)\cdot (y,v))&=\phi(x\cdot y + u\cdot y + x\cdot v,u\cdot v)\\
  &=x\cdot y + u\cdot y + x\cdot v + u\cdot v\\
  &=(x+u)\cdot (y+v).
    \end{align*}
So, $\phi$ is a ring homomorphism. We now construct a function
    \[\psi\colon R\to \m\sdprod \kappa\]
    by defining  $z\mapsto (z-j([z]),j([z]))$. To see that such a function is well defined, consider $z\in R$. Since $j$ (actually, $\iota\circ j$) gives a section to the canonical projection $\pi\colon R\to R/\m$,
    \[\pi(z-j([z]))=\pi(z)-\pi(j([z]))=[z]-[z]=0.\]
    and thus, $z-j([z])\in \ker(\pi)=\m$. It is straightforward to show that $\psi$ is a ring homomorphism, using the fact that $j$ is so and that $\lambda$ and $\rho$ are simply given by multiplication in $R$.
 Finally, we show that $\phi$ and $\psi$ are inverses of each other. Of course $\phi\circ \psi=\operatorname{id}$. Given  $(x,u)\in \m\sdprod \kappa$,
    $$\psi(\phi(x,u))=(x+u-j([x+u]),j([x+u])).$$
    To show that $j([x+u])=u$, it suffices to show that $[x+u]=j^{-1}(u)$. Since $\pi\circ \iota\circ j=\operatorname{id}$ is, in particular, injective, it suffices to show that $\pi(j([x+u]))=\pi(j(j^{-1}(u)))$. But $\pi\circ \iota\circ j=\operatorname{id}$ means to show that $[x+u]=[u]$, which is true as $x\in \m$. So $j([x+u])=u$, and we conclude that $\psi\circ \phi=\operatorname{id}$.
\end{proof}

\begin{remark}\label{remphi}
    The first part of the proof of (2) of Theorem \ref{theorsdprod} does not use the fact that $\m$ and $\kappa$ exhibit property $(\star)$ for $R$. So, every time we have $\m\in \Max{R}$ and a field $\kappa\cont R$, we can construct the semidirect product $\m\sdprod \kappa$, with $\lambda$ and $\rho$ simply given by multiplication in $R$. We will denote such a semidirect product as $\m\sdprod \kappa$ (without specifying $\lambda$ and $\rho$ in the notation of the semidirect product). Moreover, we always have the ring homomorphism
    $$\phi\colon\m\sdprod \kappa\to R$$
    that sends $(x,u)$ to $x+u$.
\end{remark}

The purpose of the next theorem is to aid in identifying two classes of rings that satisfy 
$(\star)$. However, we first need a lemma.

\begin{lemma}\label{propphiinj}
Let $R$ be a ring, $\m\in \Max{R}$ and $\kappa$ be a subfield of $R$. Then the ring homomorphism $\phi$ of Remark \ref{remphi} is injective, so that $\m\sdprod \kappa\cont R$.
\end{lemma}

\begin{proof}    
    Let $(x,u),(y,v)\in \m\sdprod \kappa$ such that
    $$x+u=\phi(x,u)=\phi(y,v)=y+v.$$
    Then $x-y=v-u$. But $x-y\in \m$ and $v-u\in \kappa$. Since $\m\cap \kappa=\{0\}$, 
    we must have $x-y=0=v-u$. So $(x,u)=(y,v)$.
\end{proof}

\begin{theorem}\label{theorsum}
    Let $R$ be a ring, $\m\in \Max{R}$ and $\kappa\cont R$ be a field. The following properties are equivalent:
\begin{enumerate}
\item The inclusion $\phi:\m\sdprod \kappa\cont R$ is an isomorphism, which is equivalent to $\m$ and $\kappa$ exhibiting property $(\star)$ for $R$ (thanks to Theorem \ref{theorsdprod}).

\item For every element $a\in R$, there exist $x\in \m$ and $u\in \kappa$ such that $a=x+u$, which is equivalent to $R\iso \m\oplus \kappa$ as additive groups.
\end{enumerate}
\end{theorem}

\begin{proof}
    By Lemma \ref{propphiinj}, the ring homomorphism $\phi\colon\m\sdprod \kappa \to R$ is injective. Then $\phi$ is an isomorphism if and only if it is a surjective function, which means that every element $a\in R$ is expressible as $\phi(x,u)$ for some $(x,u)\in \m\sdprod \kappa$. By construction of $\phi$, this holds if and only if every element $a\in R$ is expressible as $x+u$ with $x\in \m$ and $u\in \kappa$.
\end{proof}

\begin{corollary}
    If $\m\in \Max{R}$ and a field $\kappa\cont R$ exhibit property $(\star)$ for a ring $R$, then $\m$ and any other field $\kappa'$ such that $\kappa\cont \kappa'\cont R$ exhibit property $(\star)$ for $R$. Hence, $\kappa$ is always a maximal field (up to isomorphism) contained in $R$ .
\end{corollary}
\begin{proof}
    Assume that $\m\in \Max{R}$ and $\kappa\cont R$ exhibit property $(\star)$ for $R$, and let $k'$ be a field such that $\kappa\cont \kappa'\cont R$. By Theorem \ref{theorsum}, every element $a\in R$ is expressible as a sum $x+u$ with $x\in \m$ and $u\in \kappa$. But notice that $x\in \kappa'$, since $\kappa\cont \kappa'$. So, by using  Theorem \ref{theorsum} again, we obtain that $\m$ and $\kappa'$ exhibit property $(\star)$ for $R$.
    By Remark \ref{remuniquenessmk}, fixed $\m$, there can only be one field up to isomorphism that exhibit property $(\star)$ for $R$. So $\kappa\iso \kappa'$. As $\kappa'$ was arbitrary, we conclude that $\kappa$ is a maximal field (up to isomorphism)  in $R$.
\end{proof}

\begin{theorem}\label{theorAB}
    The following two classes of rings satisfy property $(\star)$:
    \begin{enumerate}
        \item[(A)] Local rings $(R,\m)$ such that every invertible element of $R$ can be written as $x+u$ with $x\in \m$ and $u\in \kappa$, a fixed field $\kappa\cont R$;
        \item[(B)] Rings $R$ such that $U(R)\cont \kappa$ with $\kappa$ a subfield of $R$ and every non-invertible element of $R$ can be written as  $x+u$ with $x$ in a fixed maximal ideal $\m$ of $R$ and $u\in \kappa$.
    \end{enumerate}
\end{theorem}
\begin{proof}
Let $(R,\m)$ be a local ring belonging to class (A). Since every non-invertible element $a$ of $R$ is in $\m$, it can be written as a sum $x+u$ with $x\in \m$ and $u\in K$ by simply taking $x=a$ and $u=0$. Moreover, by assumption, also every invertible element $a$ of $R$ can be written as $x+u$ with $x\in \m$ and $u\in K$. So, $R$ satisfies property $(\star)$ by Theorem \ref{theorsum}.

Suppose now $R$ be a ring belonging to class (B). Then, by assumption, every non-invertible element $a$ of $R$ can be written as  $x+u$ with $x\in \m$ and $u\in K$. Moreover, every invertible element $a$ of $R$ is contained in $\kappa$, and thus, it can be written as  $x+u$ with $x\in \m$ and $u\in K$ by simply taking $x=0$ and $u=a$. So, $R$ satisfies property $(\star)$ by Theorem \ref{theorsum}.
\end{proof}

\begin{remark}
$\U(R)\cont \kappa$ if and only if $\U(R)\cup \{0\}=\kappa$. Indeed, all non-zero elements of the field $\kappa$ are invertible.
\end{remark}

The following proposition shows that the classes (A) and (B) are essentially disjoint.

\begin{proposition}\label{propAB}
    The  rings that belong to both classes (A) and (B) are only  the fields. 
\end{proposition}

\begin{proof}
It is trivial to check that every field is in class (A) and in class (B).  Let us suppose now $R$ be a ring that is in class (A) and in class (B), and let $\m$ be the unique maximal ideal of $R$. Consider an element $a\in \m$and suppose $a\neq 0$.  Observe that $a+1$ is either invertible, or non-invertible in $R$. But if $a+1$ is invertible, then by (B), it must be $a+1\in \kappa$, which implies that $a\in \kappa$, a contradiction. If $a+1$ is not invertible, it must be $a+1\in \m$, which implies that $1\in \m$, once again, a contradiction. Hence, we must have $a=0$, and so $\m=\{0\}$, which implies that $R$ is a field.
 \end{proof}

\begin{proposition}\label{expr}
All equicharacteristic complete local Noetherian commutative rings belong to class~(A). 
\end{proposition}
\begin{proof}
    Let $R$ be an equicharacteristic complete local Noetherian commutative ring. By Cohen structure theorem (Theorem \ref{C46}), $R$ satisfies property $(\star)$, exhibited by some $\m\in \Max{R}$ and a field $\kappa\cont R$. By Theorem \ref{theorsum}, every element $a\in R$ is expressible as a sum $x+u$ with $x\in \m$ and $u\in \kappa$. In particular, this holds for all elements $a\in \operatorname{U}(R)$. Moreover $A$ is in particular local.
\end{proof}

\begin{remark}
Thanks to Example \ref{exlockx} and Proposition \ref{expr}, we will see that the set of (A) type rings is strictly larger than the set of rings that satisfy Cohen structure theorem's assumptions.

The set of (A) type rings is also strictly larger than the set of rings that satisfy the assumptions of Nagata's generalization of Cohen structure theorem (\cite{Nag50}). Indeed, this holds by the same argument of the proof of Proposition \ref{expr}.
\end{remark}

\begin{remark}
    For a given local ring 
    $R$, it is generally much easier to verify whether it belongs to class (A) than to determine whether it is complete. More broadly, our results provide a simpler and more efficient method for checking whether a given ring satisfies property  $(\star)$.
\end{remark}

\begin{remark}
    The classes (A) and (B) are not exhaustive. Example \ref{exfunctions} shows that $\operatorname{Hom}(X,\kappa)$ satisfies property $(\star)$ without being in (A) or in (B). Indeed, it is not local (as long as $X$ is not a singleton set), as $\m_a$ is a maximal ideal for every $a\in X$. And every function which does not vanish anywhere is invertible, so that $\operatorname{Hom}(X,\kappa)$ has far more invertible functions than the non-zero constant functions.
    However, most of our examples of rings that satisfy property $(\star)$ are in class (A) or class (B), and it is quite easy to check if a ring belongs to class (A) or class (B). In fact, it is quite easy to check if a ring satisfies condition (2) of Theorem \ref{theorsum}, to understand whether it satisfies property $(\star)$ or not. 
\end{remark}

\section{Applications and further examples}

In this section, we present some  applications of the results we proved in Section \ref{secmainresults}. Moreover, we present some methods to construct examples of rings that satisfy property $(\star)$. One of such methods is given by localizing rings that satisfy property $(\star)$. Another one is an application of Theorem \ref{theorsdprod}, which says that all semidirect products of a not necessarily unitary ring and a field are rings that satisfy property $(\star)$. In fact, the following general result holds, which says that we can take semidirect products with any ring $R$ that satisfies property $(\star)$, rather than just with fields.

\begin{proposition}\label{propsr}
    Let $R$ be a ring that satisfies property $(\star)$, exhibited by $\m\in \Max{R}$ and a field $\kappa\cont R$. Then for every ring $S$ and for every $\lambda$ and $\rho$ as in Definition \ref{defsdprod} (so any $S$ such that a semidirect product with $R$ exists), we have that $S\sdprod[\lambda,\rho] R$ satisfies property $(\star)$. More precisely, property $(\star)$ for $S\sdprod[\lambda,\rho] R$ is exhibited by $S\sdprod[\lambda',\rho'] \m$, with $\lambda'$ and $\rho'$ given by the restrictions of $\lambda$ and $\rho$ to $\m$, and the same field $\kappa$.
\end{proposition}
\begin{proof}
    Since $R$ satisfies property $(\star)$, by Theorem \ref{theorsdprod}, $R\iso \m\sdprod \kappa$. We prove that
    $$S\sdprod[\lambda,\rho] R\iso S\sdprod[\lambda,\rho] (\m\sdprod \kappa)\iso (S\sdprod[\lambda',\rho'] \m)\sdprod[\lambda'',\rho''] \kappa,$$
    for some $\lambda',\rho',\lambda'',\rho''$. We define $\lambda'$ and $\rho'$ to be the restrictions of $\lambda$ and $\rho$ to $\m\cont R$. Then we immediately see that $\lambda'$ and $\rho'$ satisfy all the conditions of Definition \ref{defsdprod} to ensure that we can construct the semidirect product $S\sdprod[\lambda',\rho'] \m$. Consider now $\ell'$ and $r'$ the restrictions of $\lambda$ and $\rho$ to $\kappa\cont R$. We define
    \[\lambda''\colon\kappa\to \End[\Mod{(S\sdprod[\lambda',\rho'] \m)}]{(S\sdprod[\lambda',\rho'] \m)_{(S\sdprod[\lambda',\rho'] \m)}}\vspace*{0.2ex}\]
    by sending $u\in \kappa$ to the endomorphism
    $$(x,y)\in S\sdprod[\lambda',\rho'] \m \mapsto (\ell'(u)(x),u\cdot_R y)=(\lambda(u)(x),u\cdot_R y)\in S\sdprod[\lambda',\rho'] \m.$$
    It is straightforward to show that such endomorphism is linear and that $\lambda$ is a ring homomorphism, using that $\lambda',\rho'$ as well as $\ell',r'$ are restrictions of $\lambda,\rho$ and that $\lambda,\rho$ satisfy the conditions of Definition \ref{defsdprod}. Analogously, we define 
    \[\rho''\colon\kappa\to \End[\Mod{(S\sdprod[\lambda',\rho'] \m)}]{_{(S\sdprod[\lambda',\rho'] \m)}(S\sdprod[\lambda',\rho'] \m)}\vspace*{0.2ex}\]
    by sending $u\in \kappa$ to the endomorphism
    $$(x,y)\in S\sdprod[\lambda',\rho'] \m \mapsto (r'(u)(x),y\cdot_R u)=(\rho(u)(x),y\cdot_R u)\in S\sdprod[\lambda',\rho'] \m,$$
    and $\rho''$ is an anti-homomorphism. It is  straightforward to show that the coherence conditions of Definition \ref{defsdprod} hold. We can thus construct the semidirect product $(S\sdprod[\lambda',\rho'] \m)\sdprod[\lambda'',\rho''] \kappa$.
It remains to prove that \[S\sdprod[\lambda,\rho] (\m\sdprod \kappa)\iso (S\sdprod[\lambda',\rho'] \m)\sdprod[\lambda'',\rho''] \kappa.\] Of course, their additive groups are isomorphic. A straightforward calculation also shows that  the multiplications coincide, so that we have a ring isomorphism.
By Theorem \ref{theorsdprod}, we conclude that $S\sdprod[\lambda',\rho'] \m\in \Max{S\sdprod[\lambda,\rho] R}$ and $\kappa\cont S\sdprod[\lambda,\rho] R$ exhibit property $(\star)$ for $S\sdprod[\lambda,\rho] R$.
\end{proof}

\begin{corollary}\label{corollsdprod}
    Given any $\kappa$-algebra $A$ with $\kappa$ a field, the semidirect product $A\sdprod[\lambda,\rho] \kappa$ (with $\lambda$ and $\rho$ given by the $\kappa$-algebra structure of $A$) satisfies property $(\star)$.
In particular, $\kappa\sdprod \kappa$ always satisfy property $(\star)$. We denote $\kappa\sdprod \kappa$ without specifying $\lambda$ and $\rho$, which are simply given by multiplication in $\kappa$. 
\end{corollary}
\begin{proof}
    Given any $\kappa$-algebra $A$ with $\kappa$ a field, we can define $\lambda$ and $\rho$ to be given by the $\kappa$-algebra structure of $A$, so that we can construct the semidirect product $A\sdprod[\lambda,\rho] \kappa$. Viewing the $\kappa$-algebra structure as given by an injective homomorphism $\kappa\to A$, we get that $\lambda$ and $\rho$ are simply given by multiplication in $A$. Since $\kappa$ has property $(\star)$, being a field, by Proposition \ref{propsr} (or Theorem \ref{theorsdprod}), we obtain that $A\sdprod[\lambda,\rho] \kappa$ has property $(\star)$.
\end{proof}

\begin{remark}\label{remkk}
    We describe in detail the structure of $\kappa\sdprod \kappa$ with $\kappa$ a field. By Corollary \ref{corollsdprod}, this is a class of examples of rings that satisfy property $(\star)$. The additive group of $\kappa\sdprod \kappa$ is $\kappa\oplus \kappa$. The multiplication is given by
    $$(a,b)\cdot (c,d)=(a\cdot c+b\cdot c+a\cdot d,b\cdot d).$$
    The ideal $\set{(x,0)}{x\in \kappa}\in \Max{\kappa\sdprod \kappa}$ and the field $\kappa\iso \set{(0,y)}{y\in \kappa}\cont \kappa\sdprod \kappa$ exhibit property $(\star)$ for $\kappa\sdprod \kappa$.
We now characterize the invertible elements $(a,b)$ of $\kappa\sdprod \kappa$. In order to have
    $$(a,b)\cdot (c,d)=(a\cdot c+b\cdot c+a\cdot d,b\cdot d)=(0,1),$$
    we need to have $b\neq 0$, so that we can take $d=b^{-1}$, and either $a=0$ or $a\neq -b$, so that we can take $c=-(a+b)^{-1}\cdot a\cdot b^{-1}$. If $b\neq 0$ and $a=0$, then $(c,b^{-1})\cdot (0,b)=(0,1)$, which forces $c=0$. If $b\neq 0$, $a\neq 0$, and $a\neq -b$, a calculation shows that also $(-(a+b)^{-1}\cdot a\cdot b^{-1},b^{-1})\cdot (a,b)=(0,1)$. We conclude that
    $$U(\kappa\sdprod \kappa)=\set{(a,b)\in \kappa\sdprod \kappa}{b\neq 0 \text{ and either } a=0 \text{ or } a\neq -b}.$$
\end{remark}

\begin{proposition}\label{propnotlocal}
     If $B$ is a unitary ring with $1\neq 0$ and $\kappa$ is a field, then $B\sdprod[\lambda,\rho] \kappa$ is not local, for any choice of $\lambda$ and $\rho$ as in Definition \ref{defsdprod}. $B\sdprod[\lambda,\rho] \kappa$ is then a non-local ring that satisfies property $(\star)$, 
\end{proposition}

\begin{proof}
    By Theorem \ref{theorsdprod}, $\set{(x,u)\in B\sdprod[\lambda,\rho] \kappa}{u=0}$ is a maximal ideal in $B\sdprod[\lambda,\rho] \kappa$. But $(-1,1)$ is not invertible, although it does not belong to such maximal ideal. Indeed, $(-1,1)\cdot (x,u)=(0,1)$ would require in particular $u=1$. But
    $$(-1,1)\cdot (x,1)=(-x+\lambda(1)(x)+\rho(1)(-1),1)=(-x+x-1,1)=(-1,1)\neq (0,1).$$
    Thus $(-1,1)$ is not invertible.
\end{proof}

The following examples are based on  Theorem \ref{theorsdprod}, Proposition \ref{propsr}, and Corollary \ref{corollsdprod}.
    
\begin{example}
    Consider the semidirect product $R:=(\mathds{Z}/2\mathds{Z})\sdprod[\lambda, \rho] (\mathds{Z}/2\mathds{Z})$, where $\lambda$ and $\rho$ are given by multiplication in $\mathds{Z}/2\mathds{Z}$. This coincides with $\lambda$ and $\rho$ being defined by sending $1\in \mathds{Z}/2\mathds{Z}$ to the identity endomorphism of $\mathds{Z}/2\mathds{Z}$. The additive structure of  $R$ is that of $(\mathds{Z}/2\mathds{Z})\times (\mathds{Z}/2\mathds{Z})=\{(0,0),(0,1),(1,0),(1,1)\}$, whereas, the multiplicative structure is described by the following table:
\begin{table}[h!]

\centering
\begin{tabular}{c | c c c c}
& (0,0) & (0,1) & (1,0) & (1,1)\\
\hline 
(0,0) & (0,0) & (0,0) & (0,0) & (0,0)\\
(0,1) & (0,0) & (0,1) & (1,0) & (1,1)\\
(1,0) & (0,0) & (1,0) & (1,0) & (0,0)\\
(1,1) & (0,0) & (1,1) & (0,0) & (1,1)\\
\end{tabular}
\end{table}

Recall from Definition \ref{defsdprod} that the unit for this multiplication is the pair $(0,1)$. By Corollary \ref{corollsdprod} (or Theorem \ref{theorsdprod}), $(\mathds{Z}/2\mathds{Z})\sdprod[\lambda, \rho] (\mathds{Z}/2\mathds{Z})$ satisfies property $(\star)$. This is exhibited by the maximal ideal $\langle (1,0)\rangle=\{(0,0),(1,0)\}$ and the field \[\kappa=\{(0,0), (0,1)\}\iso \mathds{Z}/2\mathds{Z}\] contained in $R$. 
We observe that $R$ is not local since the ideal $\langle (1,1)\rangle=\{(1,1), (0,0)\}$ is also maximal. 
Notice that the only invertible element of $R$ is the unit $(0,1)\in \kappa$, and hence, $R$ belongs to class (B) of Theorem \ref{theorAB}.
\end{example}

\begin{example}
    Consider now the semidirect product $(\mathds{Z}/3\mathds{Z})\sdprod (\mathds{Z}/3\mathds{Z})$, where $\lambda$ and $\rho$ are given by multiplication in $\mathds{Z}/3\mathds{Z}$. By Corollary \ref{corollsdprod}, this ring satisfies property $(\star)$. However, by Proposition \ref{propnotlocal}, $(\mathds{Z}/3\mathds{Z})\sdprod (\mathds{Z}/3\mathds{Z})$ is not local. By Remark \ref{remkk}, its invertible elements are
    \[U((\mathds{Z}/3\mathds{Z})\sdprod (\mathds{Z}/3\mathds{Z}))=\{(0,1),(0,2),(1,1),(2,2)\},\]
    whence $U((\mathds{Z}/3\mathds{Z})\sdprod (\mathds{Z}/3\mathds{Z}))\not \subseteq \mathds{Z}/3\mathds{Z}$. We conclude that $(\mathds{Z}/3\mathds{Z})\sdprod (\mathds{Z}/3\mathds{Z})$ is a ring that satisfies property $(\star)$ without being in class (A) or in class (B).
\end{example}

The following is an example in which the semidirect product maps $\lambda$ and $\rho$ are not given by a $\kappa$-algebra structure.

\begin{example}
 Consider the semidirect product $R:=2(\mathds{Z}/4\mathds{Z})\sdprod[\lambda,\rho] (\mathds{Z}/2\mathds{Z})$, where $\lambda$ and $\rho$ are defined by sending $1\in \mathds{Z}/2\mathds{Z}$ to the identity endomorphism of $2(\mathds{Z}/4\mathds{Z})$. This is well defined since the period of such identity endomorphism is $2$, which divides the period of $1$ in $\mathds{Z}/2\mathds{Z}$. Notice that $2(\mathds{Z}/4\mathds{Z})$ is not a $\mathds{Z}/2\mathds{Z}$-algebra since there are no injections of $\mathds{Z}/2\mathds{Z}$ into $2(\mathds{Z}/4\mathds{Z})$. So $\lambda$ and $\rho$ are not given by an algebra structure. The additive structure of  $R$ is that of $2(\mathds{Z}/4\mathds{Z})\times (\mathds{Z}/2\mathds{Z})=\{(0,0), (0,1), (2,0), (2,1)\}$, whereas the multiplicative structure is described by the table below.
\begin{table}[h!]
\centering
\begin{tabular}{c | c c c c}
& (0,0) & (0,1) & (2,0) & (2,1)\\
\hline 
(0,0) & (0,0) & (0,0) & (0,0) & (0,0)\\
(0,1) & (0,0) & (0,1) & (2,0) & (2,1)\\
(2,0) & (0,0) & (2,0) & (0,0) & (2,0)\\
(2,1) & (0,0) & (2,1) & (2,0) & (0,1)\\
\end{tabular}
\end{table}

\noindent By Proposition \ref{propsr}, $2(\mathds{Z}/4\mathds{Z})\sdprod[\lambda,\rho] (\mathds{Z}/2\mathds{Z})$ satisfies property $(\star)$. This is exhibited by the maximal ideal $\langle (2,0)\rangle=\{(0,0),(2,0)\}$ and the field $\kappa=\{(0,0),(0,1)\}\iso \mathds{Z}/2\mathds{Z}$ contained in $R$. Notice that $2(\mathds{Z}/4\mathds{Z})\sdprod[\lambda,\rho] (\mathds{Z}/2\mathds{Z})$ is local since the maximal ideal $\m$ contains all the non-invertible elements. 
\end{example}

\begin{remark}
The only possible semidirect products of the form $\mathds{Z}/n \mathds{Z}\sdprod[\lambda,\rho] \mathds{Z}/p \mathds{Z}$ with $n\in \mathds{Z}_+$ and $p$ a prime number are the ones with $n=p$, where $\lambda$ and $\rho$ are simply given by multiplication in $\mathds{Z}/p \mathds{Z}$ (or the trivial ones $0\sdprod \mathds{Z}/p \mathds{Z}$). Indeed, in order to define a (unitary) semidirect product, we need a ring homomorphism \[\lambda\colon\mathds{Z}/p \mathds{Z}\to \End[\Mod{\mathds{Z}/n \mathds{Z}}]{\mathds{Z}/n \mathds{Z}_{\mathds{Z}/n \mathds{Z}}}\] that preserve 1. So, we need to send 1 to the identity endomorphism. But then the period of the identity endomorphism must divide the period of 1 in $\mathds{Z}/p \mathds{Z}$, which means that $n$ divides $p$. There are then two possibilities: either $n=1$ and $\mathds{Z}/n \mathds{Z}=\{0\}$ or $n=p$. When $n=p$, sending $1$ to the identity endomorphism, then it completely determines $\lambda$, since $\mathds{Z}/p \mathds{Z}$ is cyclic. It is easy to see that such $\lambda$ is given by taking multiplications in $\mathds{Z}/p \mathds{Z}$, and analogously for $\rho$.
\end{remark}

The following result yields numerous examples of rings that satisfy property $(\star)$. 

\begin{proposition}\label{proplochasstar}
Let $R$ be a ring that satisfies property $(\star)$, exhibited by $\m\in \Max{R}$ and a field $\kappa\cont R$. Assume then that $R\setminus \m$ does not contain any zero divisors 
Then the localization $R_\m$ of $R$ at $\m$ satisfies property $(\star)$. More precisely, $\m R_\m$ is the only maximal ideal of $R_\m$ and 
$$R_\m/(\m R_\m)\iso \kappa\cont R\cont R_\m.$$
\end{proposition}
\begin{proof}
Thanks to Theorem \ref{theorsdprod} and Theorem \ref{theorsum}, it suffices to show that for every $\dfrac{a}{b} \in R_\m$, there exist $\dfrac{x}{c} \in \m R_\m$ and $u\in \kappa$ such that $\dfrac{a}{b}=\dfrac{x}{c}+\dfrac{u}{1}$. Since $R$ has property $(\star)$, by Theorem \ref{theorsdprod} and Theorem \ref{theorsum}, there exist $y$, $z\in \m$ and $v$, $w\in \kappa$ such that $a=y+v$ and $b=z+w$. Moreover $w\neq 0$ since $b\in R\setminus \m$. We notice that $\dfrac{a}{b} -\dfrac{v}{b}=\dfrac{y}{b} \in\m R_\m$ and so it suffices to write $\dfrac{v}{b}$ as a sum of an element of $\m R_\m$ and an element of $\kappa$. We now observe that
\[\dfrac{w}{b}=\dfrac{b}{b}+\dfrac{w}{b}-\dfrac{b}{b}=\dfrac{b}{b}+\dfrac{w-z-w}{b}=1-\dfrac{z}{b},\]
and $\dfrac{z}{b} \in\m R_\m$. Since $w$ is invertible, we obtain
\[\dfrac{1}{b}=\dfrac{w^{-1}}{1}\cdot \dfrac{w}{b}=\dfrac{w^{-1}}{1}+\dfrac{w^{-1}}{1}\cdot 1-\dfrac{z}{b},\]
and $\dfrac{w^{-1}}{1}\cdot 1-\dfrac{z}{b}\in \m R_{\m}$. So, there exists $f\in \m R_{\m}$ such that $\dfrac{v}{b}=v\cdot w^{-1} + f$, and hence, $\dfrac{a}{b}$ can be written as a sum of an element in $\m R_\m$ and an element in $\kappa$.
\end{proof}

\begin{remark}
As a consequence, all integral domains that satisfy property $(\star)$ are contained in a ring of class (A).
\end{remark}

\begin{example}\label{exlockx}
   Consider the ring $\kappa[x]$ of polynomials with coefficients in a field $\kappa$ and its localization at its maximal ideal $\langle x\rangle$:
\[
R := \mathds{\kappa}[x]_{\langle x \rangle} = \left\{ \dfrac{f(x)}{g(x)} \in \mathds{\kappa}(x) \mid x \text{ does not divide } g(x) \right\} \subset \mathds{\kappa}(x).
\]
Since $\kappa[x]$ has property $(\star)$, by Proposition \ref{proplochasstar}, $R=\mathds{\kappa}[x]_{\langle x \rangle}$ has also property $(\star)$, exhibited by $\kappa$ and the only maximal ideal $\langle x \rangle R$ of $R$.

We notice that \( \kappa[x]_{\langle x\rangle} \) is not complete. Indeed, since $\kappa[x]$ is Noetherian and $\langle x\rangle$ is a maximal ideal of $\kappa[x]$, it is well known that the completion of \( \kappa[x]_{\langle x\rangle} \) coincides with the completion of $\kappa[x]$, which is the ring $\kappa\llbracket x \rrbracket$ of formal power series, and not all formal power series are rational (i.e.\ expressible as a fraction of polynomials).
\end{example}

\begin{remark}
    Example \ref{exlockx} shows that the class (A) of rings considered in Proposition \ref{propAB} strictly generalizes both the classes of rings  satisfying the assumptions of Cohen structure theorem and rings  satisfying the assumptions of Nagata's generalization.
\end{remark}

\end{document}